\pdfoutput=1
\documentclass[11pt]{amsart}
\usepackage[utf8]{inputenc}
\usepackage[margin=1in, headheight=14.5pt]{geometry}
\usepackage{parskip}

\usepackage{calligra}
\usepackage{mathrsfs}
\usepackage[mathscr]{euscript}
\usepackage{bbm}

\usepackage{enumitem}
\usepackage{hyperref}
\usepackage[nameinlink]{cleveref}
\usepackage{url} 
\usepackage[style=alphabetic,url=false,backref=true]{biblatex}

\usepackage{subcaption}
\usepackage{pgfplots}
\pgfplotsset{compat=1.18}
\usepackage{xcolor}
\usepackage{tikz-cd}

\newcommand{\mc}[1]{\mathcal{#1}}
\newcommand{\mf}[1]{\mathfrak{#1}}

\newcommand{\bb}[1]{\mathbb{#1}}
\newcommand{\mbbm}[1]{\mathbbm{#1}}
\newcommand{\olin}[1]{\overline{#1}}
\newcommand{\ulin}[1]{\underline{#1}}
\newcommand{\wt}[1]{\widetilde{#1}}

\newcommand{\Z}{\mathbb{Z}}
\newcommand{\N}{\mathbb{N}}
\newcommand{\R}{\mathbb{R}}

\newcommand{\inv}{^{-1}}

\newcommand{\vphi}{\varphi}

\let\hom\relax
\DeclareMathOperator{\hom}{Hom}

\DeclareMathOperator{\ini}{in}
\DeclareMathOperator{\LT}{in}

\newcommand{\ca}[1]{\mathscr{#1}}
\newcommand{\pstab}[1]{\wt{#1}}

\newcommand{\sysalg}[1]{\ca C^{#1}}
\newcommand{\cc}[2]{A_{#1}(#2)}

\DeclareMathOperator{\vol}{vol}
\DeclareMathOperator{\cmult}{vol}

\newtheorem{theorem}{Theorem}[section]
\newtheorem{thm}[theorem]{Theorem}
\newtheorem{lemma}[theorem]{Lemma}
\newtheorem{prop}[theorem]{Proposition}
\newtheorem{cor}[theorem]{Corollary}
\newtheorem{fact}[theorem]{Fact}
\newtheorem{conj}[theorem]{Conjecture}

\newtheorem*{thmA}{Theorem A}
\newtheorem*{thmB}{Theorem B}
\newtheorem*{thmC}{Theorem C}

\theoremstyle{definition}
\newtheorem{defn}[theorem]{Definition}

\newtheorem{rmk}[theorem]{Remark}
\newtheorem{zb}[theorem]{Example}
\newtheorem{notation}[theorem]{Notation}

\newtheorem*{ackblock}{Acknowledgements}

\title{Cartier algebras through the lens of $p$-families}
\author{Anna Brosowsky}
\thanks{The author was supported in part by NSF grants DMS \#1840234 and \#2101075, and by NSF postdoctoral
fellowship \#2402293.}
\address{Department of Mathematics, University of Nebraska--Lincoln, Lincoln, NE 68588, USA}
\email{abrosowsky2@unl.edu}
\date{} 


\begin{document}
\begin{abstract}
We study $F$-graded systems of ideals in $R$, which are sequences of ideals giving rise to Cartier algebras on $R$. We identify how properties of these systems (or modifications of these systems) affect the singularity properties of the corresponding Cartier algebra.
In particular, we show that in a 
Gorenstein and strongly $F$-regular 
local ring, strong $F$-regularity and $F$-splitting are the same for a special class of $F$-graded systems called $p$-families.
Further, we make use of this and a new operation we introduce called \emph{$p$-stabilization} to get a criterion that in a 
Gorenstein and strongly $F$-regular 
local ring, a system is strongly $F$-regular exactly when its $p$-stabilization is $F$-split.
Finally, we associate a combinatorial object to systems built out of monomial ideals and show how this can help compute the $p$-stabilization.
\end{abstract}

\maketitle

\section{Introduction}
In positive characteristic commutative algebra, we are often interested in classifying the singularities of a ring $R$ by leveraging the power of the Frobenius map.
Two common such singularity classes are $F$-splitting (where the Frobenius splits as an $R$-module map) and strong $F$-regularity (where the Frobenius and other related maps all split as $R$-module maps). 
To put these into context, $F$-splitting and strong $F$-regularity are positive characteristic analogs of the characteristic zero notions of log canonicity and Kawamata log terminality. 
Further, strongly $F$-regular rings are Cohen-Macaulay and normal. 
As these singularity classes are defined in terms of when certain maps split, it is natural to consider what happens when further restrictions are placed on the maps which provide the splittings. 
This led to the introduction of Cartier algebras by Schwede \cite{Schwede.11a}, see \Cref{def:cartier-subalgebra}.

Cartier algebras generalize the more traditional ``pairs'' setting, and can carry information about an associated divisor, or information about a quotient ring. They can also be used to better understand the original ring.
For example, Carvajal-Rojas and Smolkin showed that if a $k$-algebra of essentially finite type is diagonally $F$-regular (meaning a certain Cartier algebra is strongly $F$-regular), 
then there is an effective bound relating symbolic powers of prime ideals to typical powers of prime ideals, namely that $\mf p^{(n\cdot \dim R)} \subseteq \mf p^n$ \cite{Carvajal-Rojas+Smolkin.20}. 
Notably, $\dim R$ is the same bound as for a regular ring \cite{Ein+etal.01,Hochster+Huneke.02,Ma+Schwede.18}.

We study sequences of ideals called \emph{$F$-graded systems} (see \Cref{def:f-graded}), primarily working over local rings which are Gorenstein and strongly $F$-regular. 
These arise naturally when considering Cartier algebras, and in particular one can define a notion of strong $F$-regularity and Frobenius splitting for $F$-graded systems (see \Cref{def:f-singularities-f-graded}).
A special class of $F$-graded systems is that of \emph{$p$-families} (see \Cref{def:p-family}), named by Hern\'andez and Jeffries \cite{Hernandez+Jeffries.18}. This class is independently of interest as $p$-families appear when defining the Hilbert-Kunz multiplicity \cite{Monsky.83} and the $F$-signature \cite{Tucker.12}.
However, not much is known about the classes of Cartier algebras these $p$-families correspond to. 
As it turns out, for $p$-families, Frobenius splitting and strong $F$-regularity collapse into the same condition:

\begin{thmA}[\Cref{p-family-f-split-iff-sfr}]
Let $(R,\mf m)$ be a Gorenstein and strongly $F$-regular $F$-finite
local ring. Let $\mf b_\bullet$ be a $p$-family in $R$. Then $\mf b_\bullet$ is $F$-split if and only if it is strongly $F$-regular.
\end{thmA}
We also describe a new operation on \mbox{$F$-graded} systems called \emph{$p$-stabilization}, which turns an \mbox{$F$-graded} system into a closely related $p$-family in a way that preserves strong $F$-regularity. In particular, when combined with the previous theorem, we can show:

\begin{thmB}[\Cref{a-SFR-iff-b-SFR}, \Cref{a-SFR-iff-b-Fsplit}]
Let $(R,\mf m)$ be a Gorenstein and strongly $F$-regular $F$-finite local ring. 
Let \(\mf a_\bullet\) be an $F$-graded system in $R$ with $\mf a_1\neq 0$, and let $\pstab{\mf a}_\bullet$ be the $p$-stabilization of $\mf a_\bullet$. 
Then $\mf a_\bullet$ is strongly $F$-regular if and only if $\pstab{\mf a}_\bullet$ is strongly $F$-regular if and only if $\pstab{\mf a}_\bullet$ is $F$-split. 
\end{thmB}

In \Cref{sec:shapes}, we focus on polynomial rings and $F$-graded systems where every constituent ideal is monomial, and use combinatorics to gain insight into the properties of the system. 
By taking advantage of the correspondence between monomials in $k[x_1,\ldots, x_d]$ and points in $\N^d$, we define an \emph{associated $p$-body} in $(\N[1/p])^d$ (see \Cref{def:associated-shape}) to a monomial $F$-graded system. Conversely, we define an \emph{associated $p$-family} to subsets of $(\N[1/p])^d$ (see \Cref{def:associated-p-family}).
This construction extends Hern\'andez and Jeffries's notion of an associated $p$-body for a $p$-family, and gives a concrete way to encapsulate the asymptotic behavior of an $F$-graded system. Further, we show that this correspondence is intimately connected to $p$-stabilization:

\begin{thmC}[{\Cref{shape-stabilization-correspondence}}]
Let $S$ be a polynomial ring over an $F$-finite field. If $\mf b_\bullet$ is a monomial $F$-graded system, then the associated $p$-family of the associated $p$-body of $\mf b_\bullet$ is the $p$-stabilization, i.e., $\mf a_\bullet^{\Delta(\mf b_\bullet)} = \pstab{\mf b}_\bullet$. 
If $\Delta\subset (\N[1/p])^d$,
then
$\Delta(\mf a^\Delta_\bullet) = \Delta+ \left(\N[1/p]\right)^d$.

In particular, this gives a correspondence between monomial $p$-stable $F$-graded systems and subsets of $(\N[1/p])^d$ which are invariant under adding $\left( \N[1/p]\right)^d$.
\end{thmC}

In \Cref{sec:shapes-zb}, we then illustrate how this can be used to actually compute the $p$-stabilization of some of the examples introduced back in \Cref{sec:stabilization-zb}.
Finally, in \Cref{sec:numerical-general}, we consider a numerical invariant of $F$-graded systems that extends the Hilbert-Kunz multiplicity, the $F$-signature, and the volume of a $p$-family. We end with some connections to the associated $p$-body and to Das and Meng's work on volumes \cite{Das+Meng.26}.

\begin{ackblock}
I would like to thank my advisor Karen Smith for all her guidance and suggestions. I would also like to thank Austyn Simpson for his detailed feedback on an earlier draft.
Thanks to Andy Gordon for the idea for the proof of \Cref{minl-system-shape},
and thanks to Javier Carvajal-Rojas for his suggestions on how to generalize to the Gorenstein case. I would like to thank Seungsu Lee and Suchitra Pande for the helpful conversations and Jack Jeffries for his feedback. Finally, thanks to my thesis committee members: Mircea Musta\c{t}\u{a}, Karen Smith, Austyn Simpson, Quentin F.~Stout, and Kevin Tucker.

A version of this work originally appeared as Chapter~4 of my dissertation, see \cite{Brosowsky.24}.

\end{ackblock}

\section{Background}

Given a ring $R$ with characteristic $p>0$,
the Frobenius map is the ring homomorphism
\[
F:R\to R \mbox{ defined by }F(r) = r^p.
\]
The \emph{Frobenius pushforward} $F_*R$ is the $R$-algebra with module structure coming from the Frobenius via restriction of scalars. Explicitly,
as a ring, $F_*R = \{F_*r\: | \: r\in R\}$ is exactly the same as $R$, just with formal symbol $F_*$ prepended everywhere. For example, multiplication is $(F_* r)(F_* s) = F_*(rs)$. However, its $R$-module structure is $rF_* s = F_* (r^p s)$.

The Frobenius map is now written as $F: R \to F_*R$ so that $F(r) = F_*(r^p)$, and the $R$-module action is now written $rF_*s = F_*(r^p s)$.
We can also iterate the Frobenius, writing $F^e: R\to F_*^e R$, where $F^e(r) = F_*^e(r^{p^e})$ and $rF_*^e s = F_*^e (r^{p^e}s)$.
Finally, the \emph{$e$-th Frobenius bracket power} of an ideal $I$ is
\[
I^{[p^e]} = (IF_*^eR) \cap R = \langle g^{p^e} \: | \: g\in I\rangle.
\]
If $I=\langle g_1,\ldots, g_t\rangle$, then $I^{[p^e]} = \langle g_1^{p^e},\ldots, g_t^{p^e}\rangle$.

Studying the $R$-module structure of $F_*^eR$ yields useful information about the singularities of $R$, for example, as in the following classic result of Kunz:

\begin{thm}[{\cite[Cor.~2.7]{Kunz.69}}]
\label{kunz-thm}
A Noetherian ring $R$ in prime characteristic is regular if and only if $F_*R$ is a flat $R$-module.
\end{thm}

Our focus will be on classes of singularities defined by the behavior of maps in $\hom_R(F_*^eR,R)$. Thus we must first consider a cohesive way to think about collections of such maps

\begin{defn}
\label{def:full-cartier-algebra}
The \emph{full Cartier algebra} on $R$ is the graded non-commutative ring 
\[
\ca C^R = \bigoplus_{e\geq 0} \hom_R(F_*^e R, R),
\]
where multiplication is as defined below in \Cref{eq:cartier-mult}.
\end{defn}

As we shall see, this multiplication is basically composition: First, given a map $\psi \in \hom_R(F_*^dR, R)$, write $F_*^e \psi: F_*^{e+d}R\to F_*^eR$ for the \emph{Frobenius pushforward} of the map, where 
\[
(F_*^e\psi)(F_*^{e+d} r) = (F_*^e \psi)(F_*^e(F_*^d r)) = F_*^e(\psi(F_*^d r)).
\]
Now define a (non-commutative) multiplication on the abelian group $\ca C^R = \bigoplus_e \hom_R(F_*^e R, R)$ as follows: given maps $\phi \in \hom_R(F_*^e R, R)$ and $\psi \in \hom_R(F_*^d R, R)$, we define their product as
\begin{equation}
\label{eq:cartier-mult}
\phi\star \psi = \phi \circ F_*^e \psi.
\end{equation}
More concretely, for any $r\in R$ we have 
\[
(\phi\star \psi) (F_*^{e+d} r) = \phi \left( F_*^e(\psi(F_*^d r))\right).
\]

Note that $F_*^0R$ here means $R$ itself as an $R$-module, so that $\ca C^R_0 = \hom_R(R,R)\cong R$.

\begin{defn}
\label{def:cartier-subalgebra}
	A \emph{Cartier algebra} $\ca D$ is a graded subring of $\ca C^R$ such that $\ca D_0 = R$. In particular, $\ca D$ has the form $\ca D = \bigoplus_e \ca D_e$ where $\ca D_e \ \subseteq \hom_R(F_*^eR, R)$ for all $e\geq 0$.
\end{defn}

\begin{zb}[{\cite[Rmk.~3.10]{Schwede.11a}}]
\label{zb:roundup-cartier-algebra}
Let $(R,\mf a^t)$ be a pair where $\mf a$ is an ideal and the formal exponent $t$ is a positive real number. Then the corresponding Cartier algebra $\ca C^{\mf a^t}$ has 
\[
\ca C^{\mf a^t}_e = \hom_R(F_*^e R, R)\star \mf a^{\lceil t(p^e-1)\rceil} 
= \left\{\sum_i \varphi_i\circ F_*^e a_i \: \big| \: \varphi_i \in \hom_R(F_*^eR,R),\ a_i \in \mf a^{\lceil t(p^e-1)\rceil}   \right\}.
\]
\end{zb}

\begin{zb}
Let $I$ be an ideal in $R$. Then we can define a Cartier algebra $\ca D\subset \ca C^R$ via
\[
\ca D_e := \big\{\mbox{maps $\psi: F_*^eR \to R$ which descend to $\overline \psi: F_*^e(R/I) \to R/I$} \big\},
\]
where ``descend'' means the following diagram commutes:
\begin{center}
\begin{tikzcd}[ampersand replacement =\&]
F_*^eR \rar["\psi"]\dar \& R\dar \\
F_*^e(R/I)\rar["\overline \psi"] \& R/I
\end{tikzcd}
\end{center}
\end{zb}

The following useful result will allow us to give a more explicit description of the above example in the case that $R$ is a regular local ring which is $F$-finite (see \Cref{def:cartier-ver-f-sing}):
\begin{lemma}[{Fedder's Lemma, \cite[Lemma~2.1]{Glassbrenner.96}}]
\label{qr-hom-presentation}
Let $R$ be an $F$-finite regular local ring and consider ideal $I\subset R$. Then
\[
\hom_{R/I}(F_*^e (R/I), R/I) \cong F_*^e\left( \frac{I^{[p^e]}:_RI}{I^{[p^e]}}  \right)
\]
as $R/I$-modules. 
\end{lemma}

\begin{zb}[Fedder's Lemma, rephrased]
\label{zb:fedder-cartier-algebra}
Let $(R,\mf m)$ be an $F$-finite regular local ring and let $I$ be an ideal. Then the Cartier algebra $\ca D$ on $R$ composed of all maps which descend from $\ca R$ to $\ca C^{R/I}$ is
\[
\ca D_e = \hom_R(F_*^eR, R) \star \left(I^{[p^e]}:I\right).
\]
\end{zb}

Schwede introduced Cartier algebras as a further generalization of the study of singularities of (divisor) pairs, which itself was motivated by studying Frobenius splitting and strong $F$-regularity for rings. The definitions below for Cartier algebras line up with the standard definitions for rings by taking $\ca D = \ca C^R$ to be the full Cartier algebra.

\begin{defn}
\label{def:cartier-ver-f-sing}
Let $R$ be a ring of prime characteristic, and let $\ca D$ be a Cartier algebra on $R$.
\begin{itemize}

\item The pair~$(R,\ca D)$ is \emph{$F$-finite} if $R$ is $F$-finite, i.e., $F_*R$ is a finite $R$-module.

\item The pair~$(R,\ca D)$ is \emph{Frobenius split} if there exists some  $e>0$ and some $\phi \in \ca D_e$ with $\phi(F_*^e 1)=1$.

\item If $c$ is an element of $R$, then the pair~$(R,\ca D)$ is \emph{eventually Frobenius split along $c$} if there exists some $e>0$ and some $\phi \in \ca D_e$ with $\phi(F_*^e c)=1$.

\item The pair~$(R,\ca D)$ is \emph{strongly $F$-regular} if it is eventually Frobenius split along every $c$ which is not in any minimal prime of $R$.
\end{itemize}
\end{defn}
\begin{defn}
\label{def:cartier-contraction}
Let $I\subset R$ be an ideal, and $\ca D$ be a Cartier algebra. The \emph{$e^{\textrm{th}}$ Cartier contraction} of $I$ with respect to $\ca D$ is
\[
\cc{\ca D_e}{I} = \{r\in R \: | \: \varphi(F_*^er)\in I \ \,\forall \varphi \in \ca D_e\}.
\]
If $\ca D = \ca C^R$, the full Cartier algebra, then we will simply write $\cc{e}{I}$. More explicitly,
\[
\cc{e}{I} = \{r\in R \: | \: \varphi(F_*^er)\in I \ \,\forall \varphi \in \hom_R(F_*^eR,R)\}.
\]
This is often denoted as $I_e(I)$.
\end{defn}

\begin{zb}
\label{zb:regular-cc}
If $(R,\mf m)$ is an $F$-finite regular local ring, then $A_e(I) = I^{[p^e]}$. We can see this using \cite[Lemma~1.6]{Fedder.83}, 
which says $(\Phi^e\star r)(F_*^eJ)\subseteq I$ if and only if $r\in I^{[p^e]}:_R J$, where $I,J$ are ideals of $R$, $r\in R$, and $\Phi^e$ is the free generator of $\hom_R(F_*^eR,R)$, see \Cref{gor-cyclic} below.
\end{zb}

We now point out some basic properties of the Cartier contraction which will be useful later, and an observation about Gorenstein rings which is necessary to show these properties.
\begin{prop}[see {\cite[Ex.~2.6]{Blickle+etal.12} and \cite[Rmk.~4.4]{Schwede.11}}]
\label{gor-cyclic}
Let $(R,\mf m)$ be an $F$-finite Gorenstein local ring. Then $\ca C^R_1 = \hom_R(F_*^eR,R) \cong F_*^eR$ as $F_*^eR$-modules, and thus there is a map $\Phi \in \ca C^R_1$ which generates $\ca C^R_1$ as an $F_*R$-module. Further, $\Phi^e = \Phi \star \cdots \star \Phi$ generates $\ca C^R_e$.
\end{prop}
\begin{prop}
\label{cc-properties}
Suppose $R$ is $F$-finite Gorenstein local, and $I,J$ are ideals of $R$. Then
\begin{enumerate}
\item $\cc{e}{\cc{f}{I}}=\cc{e+f}{I}$
\item $\cc{e}{I:J} = \cc{e}{I}:J^{[p^e]}$
\end{enumerate}
\end{prop}
\begin{proof}
For the first equality: since $R$ is Gorenstein local, $\ca C^R_{e}\star \ca C^R_f = \ca C^R_{e+f}$. Thus
	\begin{align*}
\cc{e}{\cc{f}{I}} &= \{ r \: |\: \varphi(F_*^er)\in \cc{f}{I}\;\;  \forall \varphi\in \ca C^r_e\}\\
    &= \{r\: | \: \psi(F_*^f(\vphi(F_*^e r))) \in I \;\; \forall \psi\in \ca C^R_f, \forall \vphi\in \ca C^r_e\} \\
    &= \{ r \: | \: \phi(F_*^{e+f}(r)) \in I\;\;  \forall \phi\in \ca C^R_{e+f}\} = \cc{e+f}{I}.
	\end{align*}

For the second equality:
\begin{align*}
\cc{e}{I:J} &= \{ r \:|\: \varphi(F_*^e r) \in I:J\;\; \forall \varphi \in \ca C^R_e\} \\
    &= \{ r \:|\: \varphi(F_*^e r)J \in I\;\; \forall \varphi \in \ca C^R_e\} \\
    &= \{ r \:|\: \varphi(F_*^e (rJ^{[p^e]})) \in I\;\; \forall \varphi \in \ca C^R_e\}
    = \cc{e}{I}:J^{[p^e]}.\qedhere
\end{align*}
\end{proof}

\section{\texorpdfstring{Preliminaries on $F$-graded systems}{Preliminaries on F-graded systems}}
\label{sec:f-graded-preliminaries}

\begin{defn}[{\cite[Def~3.20]{Blickle.13}}]
	\label{def:f-graded}
Let $R$ be a commutative Noetherian ring of prime characteristic $p$, and let $\mf a_\bullet = \{\mf a_n\}_{n\in \N}$ be a sequence of ideals. We say $\mf a_\bullet$ is an  \emph{$F$-graded system of ideals} if
\begin{enumerate}
 \item $\mf a_0=R$, and
\item  
$\mf a_e^{[p^f]}\mf a_f \subseteq \mf a_{e+f}$ for all $e,f\geq 1$. 
\end{enumerate}
We will refer to $\mf a_e$ as the \emph{degree $e$ piece} (or \emph{degree $e$ ideal}) of the system.
\end{defn}
Blickle first introduced a version of this definition in \cite[Def~3.20]{Blickle.13}. 
We use the definition as it later appears in \cite[Def~4.7]{Blickle+etal.12}.

\begin{zb}[Three main examples]
\label{zb:basic-three}
For any ring $R$ of prime characteristic $p$ and any fixed ideal $J$ in $R$, the following three systems of ideals are all $F$-graded:
\begin{enumerate}
\item Setting $\mf a_e = \prod_{i=0}^{e-1} J^{[p^i]}$ for $e>0$.
\item Setting $\mf b_e = J^{[p^e]}:J$ for $e>0$.
\item Setting $\mf c_e = J^{\lceil t (p^e-1)\rceil}$ for all $e> 0$, where $t\in \R_{>0}$ is fixed.
\end{enumerate}
\end{zb}
\begin{proof}
For the first system, we will show by induction on $n$ that $\mf a_n \supset \mf a_{n-e}^{[p^e]}\mf a_e$ for all $0\leq e\leq n$.
The base cases of $n=0$ and $n=1$ are clear. Now suppose we want to prove the statement for $n+1$. If $e=0$ or $e=n+1$, again the result is clear. Otherwise, 
\[
\mf a_{n+1-e}^{[p^e]}\mf a_e 
    = \left(\prod_{i=0}^{n-e} J^{[p^i]}\right)^{[p^e]} \left(\prod_{j=0}^{e-1} J^{[p^j]}\right)
    = \left(\prod_{i=0}^{n-e} J^{[p^{i+e}]}\right) \left(\prod_{j=0}^{e-1} J^{[p^j]}\right)
    = \prod_{i=0}^{n} J^{[p^{i}]} = \mf a_{n+1}.
\]

For the second system, note that $(J^{[p^e]}:J)^{[p^f]}\subseteq J^{[p^{e+f}]}:J^{[p^f]}$, since for any $x\in J^{[p^e]}:J$, we have
\[
x^{p^f}J^{[p^f]} \subseteq (xJ)^{[p^f]} \subseteq (J^{[p^e]})^{[p^f]}.
\]
This containment makes it clear that
\[
(\mf b_e^{[p^f]}\mf b_f)J \subseteq (J^{[p^{e+f}]}:J^{[p^f]})(J^{[p^f]}:J)J \subseteq J^{[p^{e+f}]},
\]
and so
\[
\mf b_e^{[p^f]}\mf b_f \subseteq J^{[p^{e+f}]}:J = \mf b_{e+f}.
\]


For the last system, since 
\[
p^f\left\lceil t(p^e-1)\right\rceil +\left\lceil t(p^f-1)\right\rceil \geq \left\lceil t(p^{e+f}-p^f)\right\rceil + \left\lceil t(p^f-1)\right\rceil \geq \left\lceil t(p^{e+f}-p^f) + t(p^f-1)\right\rceil,
\]
we immediately have
\[
\mf c_e^{[p^f]}\mf c_f \subseteq J^{p^f\lceil t(p^e-1)\rceil} J^{\lceil t(p^f-1)\rceil} \subseteq J^{\lceil t(p^{e+f}-1)\rceil} \subseteq \mf c_{e+f}.
\qedhere
\]
\end{proof}
\begin{rmk}
\label{minl-sys-containment}
A similar argument as for system $\mf a_\bullet$ above also tells us that for \emph{any} $F$-graded system~$\mf d_\bullet$, we have $\mf d_n \supset \prod_{i=0}^{n-1}\mf d_1^{[p^i]}$, and $\mf d_{en} \supset \prod_{i=0}^{n-1} \mf d_e^{[p^{ei}]}$. In particular, system $\mf a_\bullet$ above is the minimal $F$-graded system that has $\mf a_1=J$.
\end{rmk}

The following straightforward result of Blickle, Schwede, and Tucker illustrates the original motivation behind $F$-graded systems: they are a useful way to describe Cartier algebras (as defined in \Cref{def:cartier-subalgebra}).
\begin{lemma}[{\cite[Lemma~4.9]{Blickle+etal.12}}]
If $(R,\mf m)$ is an $F$-finite local ring, then every $F$-graded system of ideals $\mf a_\bullet$ of $R$ defines a Cartier algebra $\sysalg{\mf a_\bullet}$ on $R$ by setting $\sysalg{\mf a_\bullet}_e := \ca C^R_e \star \mf a_e$ for all $e\geq 0$. Furthermore, if $R$ is Gorenstein, then \emph{every} Cartier algebra $\ca D$ arises uniquely in this manner.
\end{lemma}

In light of this lemma, we have in fact seen examples two and three of \Cref{zb:basic-three} before: the system $\mf b_\bullet$ appeared in \Cref{zb:fedder-cartier-algebra} as the $F$-graded system defining a Cartier algebra on the regular local ring $R$ which is the lift of the full Cartier algebra on $R/J$, and the system $\mf c_\bullet$ appeared in \Cref{zb:roundup-cartier-algebra} as the $F$-graded system defining the Cartier algebra for the pair $(R,J^t)$.

An interesting special case of $F$-graded systems are $p$-families of ideals, which were introduced by Hern\'andez and Jeffries for a different purpose: 
\begin{defn}[{\cite[Def~5.1]{Hernandez+Jeffries.18}}]
\label{def:p-family}
A \emph{$p$-family of ideals} is a sequence of ideals~$\mf b_\bullet$ such that~$\mf b_e^{[p]}\subseteq \mf b_{e+1}$ for all~$e$.
\end{defn}
Note that $p$-families are indeed $F$-graded; iterating the definition shows $\mf b_e^{[p^f]}\subseteq \mf b_{e+f}$ for any $f\geq 1$, and so
\[
\mf b_e^{[p^f]}\mf b_f \subseteq \mf b_{e+f}I_f \subseteq \mf b_{e+f}.
\]

\begin{zb}[{\cite[Ex.~5.4--5.7]{Hernandez+Jeffries.18}}] The following systems are all examples of $p$-families:
\begin{itemize}
\item the classic example of $I^{[p^\bullet]}$.

\item (Ex.~5.4) the sequence of Cartier contractions $\cc{e}{J}$ of an ideal under the full Cartier algebra (see \Cref{def:cartier-contraction}) also gives a $p$-family. In the setting where $T$ is an $F$-finite regular local ring, $R=T/I$, and $I\subseteq J$ are ideals of $R$, these are of the form $\cc{e}{J/I} = \left(J^{[p^e]}:_T(I^{[p^e]}:_TI)\right)/I$, see \cite[Thm.~4.6]{Brosowsky.23}.

\item (Ex.~5.5) If $I_\bullet$ is a graded family of ideals (in the typical sense), then $\mf a_e := I_{p^e}$ is a $p$-family.

As a variant of this example, if $I_\bullet$ is a graded family of ideals (in the typical sense), then $\mf a_e := I_{p^e-1}$ is an $F$-graded system.

\item (Ex.~5.6) $p$-families are preserved under arbitrary termwise product, sum, and intersection; by expansion and contraction to or from another ring; and by termwise saturation with respect to a fixed ideal.

\item (Ex.~5.7) Fix $t\in \R_{>0}$ and $f\in R$, and let $\mf a_\bullet$ be a $p$-family. Then $\mf b_e := \mf a_e : f^{\lceil tp^{e}\rceil - 1}$ is also a $p$-family.
\end{itemize}
\end{zb}

\subsection{New systems from old}

We can also modify existing $F$-graded systems to get new ones. In this subsection, we observe some basic operations on $F$-graded systems as a prelude to our more detailed study of the operation of ``$p$-stabilization'' in \Cref{sec:p-stabilization}.

First, we see that in a polynomial ring, the termwise operation of taking initial ideals with respect to a fixed monomial order preserves the property of being an $F$-graded system. 
For a polynomial~$r$, we write $\LT_<(r)$ for the initial term (or simply $\LT(r)$ if the monomial order is clear). Similarly, we write $\ini_<(I)$ or simply $\ini( I)$ for the initial term ideal of a given ideal $I$. See \cite{Ene+Herzog.12} for more background.

\begin{prop}
\label{init-ideal-is-F-graded}
Let $\mf a_\bullet$ be an $F$-graded system in $S=K[x_1,\ldots, x_n]$. Fix a monomial term order. Then the system of initial ideals $\ini \mf a_\bullet$, where
\[
\ini \mf a_e = \langle \LT(r) \: | \: r\in \mf a_e\rangle,
\]
is also $F$-graded. Similarly, if $\mf b_\bullet$ is a $p$-family, then $\ini \mf b_\bullet$ is also a $p$-family.
\end{prop}
\begin{proof}
Taking initial terms commutes with products and powers, i.e., $\LT(r^q) = \LT(r)^q$ and $\LT(rs) = \LT(r)\LT(s)$. Thus our desired condition is preserved. More explicitly, if $r\in \mf a_e$ and $s\in \mf a_f$, then
\[
(\LT(r))^{p^f} \LT(s) = \LT(r^{p^f} s) \in \ini (\mf a_{e}^{[p^f]}\mf a_f) \subseteq \ini \mf a_{e+f}.
\]
Removing the $s$'s from the argument gives the $p$-family result.
\end{proof}

Likewise, the termwise operation of taking integral closure also preserves the property of being an $F$-graded system.
\begin{prop}
\label{int-closure-is-F-graded}
Let $\mf a_\bullet$ be an $F$-graded system in some ring $R$. Then the system $\olin{\mf a}_\bullet$ of termwise integral closures $\olin{\mf a}_e = \olin{\mf a_e}$ is also $F$-graded. Similarly, if $\mf b_\bullet$ is a $p$-family, then $\olin b_\bullet$ is also a $p$-family. 
\end{prop}
\begin{proof}
For any ideal $I$, we have $F^e(\olin I)F^e_*R \subseteq \olin{F^e(I)F^e_*R}$ by persistence of integral closure applied with the Frobenius \cite[Rmk~1.1.3(7)]{Huneke+Swanson.06}, so that $F_*^e\olin I^{[q]} \subseteq \olin{F_*^e I^{[q]}}$. In other words, $\olin I^{[q]} \subseteq \olin{I^{[q]}}$. This gives the $p$-family result, since $\olin{\mf b_e}^{[p]} \subseteq \olin{\mf b_e^{[p]}} \subseteq \olin{\mf b_{e+1}}$.  

It is also true \cite[Rmk~1.3.2(4)]{Huneke+Swanson.06} that $\olin I \cdot \olin J \subseteq \olin{IJ}$. Thus 
\[
\olin{\mf a_e}^{[p^f]}\olin{\mf a_f} \subseteq \olin{\mf a_e^{[p^f]}}\olin{\mf a_f} \subseteq \olin{\mf a_e^{[p^f]}\mf a_f} \subseteq \olin{\mf a_{e+f}}.\qedhere
\]
\end{proof}

Finally, it is illustrative to observe that the short-term behavior of an $F$-graded system may not be representative of the long-term behavior, in the sense that one can ``splice'' a smaller sequence together with a larger one:

\begin{lemma}[Splicing Lemma]
\label{splicing-lemma}
Let $R$ be a commutative ring of prime characteristic $p$ and let $\mf a_\bullet$ and $\mf b_\bullet$ be $F$-graded systems. If there exists an index $E$ such that $\mf a_e \subseteq \mf b_e$ for all $e\leq 2E$, then the system $\mf c_\bullet$ with
\[
\mf c_e =
\begin{cases}
\mf a_e & e\leq E \\
\mf b_e & e > E
\end{cases}
\]
is $F$-graded.
\end{lemma}
\begin{proof}
The grading condition clearly holds if $e,f> E$. It is not hard to check that the condition also holds if $e,f\leq E$ because of the containment between $\mf a_\bullet$ and $\mf b_\bullet$:
\[
\mf c_e^{[p^f]}\mf c_f \subseteq \mf a_e^{[p^f]}\mf a_f \subseteq \mf a_{e+f} \subseteq \mf b_{e+f}. 
\]
So suppose $e\leq E$ and $f>E$. Then
\begin{align*}
\mf c_e^{[p^f]}\mf c_f &= \mf a_e^{[p^f]}\mf b_f \subseteq \mf b_e^{[p^f]}\mf b_e \subseteq \mf b_{e+f} = \mf c_{e+f} \\
\mf c_f^{[p^e]} \mf c_e &= \mf b_f^{[p^e]} \mf a_e \subseteq \mf b_f^{[p^e]}\mf b_e \subseteq \mf b_{e+f} = \mf c_{e+f}. \qedhere
\end{align*}
\end{proof}


\section{\texorpdfstring{Detecting $F$-singularities of an $F$-graded system}{Detecting F-singularities of an F-graded system}}
\label{sec:singularities-f-graded}

We can now meaningfully define $F$-singularities for an $F$-graded system.
\begin{defn}
\label{def:f-singularities-f-graded}
Let $\mf a_\bullet$ be an $F$-graded system, and let $\sysalg{\mf a_\bullet} = \bigoplus_{e\geq 0} \ca C^R_e\star \mf a_e$ be the corresponding Cartier algebra.
Then $\mf a_\bullet$ is \emph{Frobenius split} if~$(R,\sysalg{\mf a_\bullet})$ is Frobenius split. 
Likewise, $\mf a_\bullet$ is \emph{strongly $F$-regular} if~$(R,\sysalg{\mf a_\bullet})$ is strongly $F$-regular, and $\mf a_\bullet$ is \emph{eventually $F$-split along} element $c$ if $(R,\sysalg{\mf a_\bullet})$ is eventually $F$-split along $c$.
\end{defn}

\begin{rmk}
Recalling \Cref{def:cartier-ver-f-sing}, we can also restate the above definition even more explicitly as follows when $\mf a_\bullet$ is an $F$-graded system on a local ring $R$. 
\begin{itemize}
\item The system $\mf a_\bullet$ is Frobenius split if there exists some $e>0$ and some $\varphi \in \hom_R(F_*^eR, R)$ and $a\in \mf a_e$ with $\varphi(F_*^ea)=1$.

\item The system $\mf a_\bullet$ is strongly $F$-regular if for every $c$ not in any minimal prime of $R$, there exists some $e>0$ and some $\varphi \in \hom_R(F_*^eR,R)$ and $a\in \mf a_e$ with $\varphi(F_*^e(ac))=1$.
\end{itemize}
In particular, the local setting means it suffices to check only maps of the form $\varphi\star a$, rather than sums of such maps.
\end{rmk}

It can be difficult to verify that a given Cartier algebra is strongly $F$-regular, especially if one does not have an explicit ``test element'' as in \cite[Thm.~3.3]{Hochster+Huneke.89a}, and the same difficulty holds for $F$-graded systems. 
However, we shall see that for $p$-families, strong $F$-regularity and $F$-splitting collapse into the same condition. 
More specifically, our main result of this section is the following:
\begin{thm}
\label{p-family-f-split-iff-sfr}
Let $(R,\mf m)$ be a Gorenstein and strongly $F$-regular $F$-finite
local ring. Let $\mf b_\bullet$ be a $p$-family in $R$. Then $\mf b_\bullet$ is $F$-split if and only if it is strongly $F$-regular.
\end{thm}

Before proving this theorem we will need some useful tools for detecting strong $F$-regularity of an $F$-graded system.
The first is a variant of Fedder's criterion for $F$-graded systems, which relates the Cartier contractions with respect to $\mf a_\bullet$ to the Cartier contractions with respect to the ring (see \Cref{def:cartier-contraction}):

\begin{lemma}[{C.f. \cite[Lemma~2.2]{Glassbrenner.96} and \cite[Lemma~4.12]{Blickle+etal.12}}]
\label{gb-F-graded}
Let $(R,\mf m)$ be an $F$-finite Gorenstein local ring, let $\mf a_\bullet$ be an $F$-graded system, and let $\sysalg{\mf a_\bullet}:=\bigoplus_e \ca C_e^R\star \mf a_e$ be the corresponding Cartier algebra. Then for any ideal $I$ and any $e>0$,
\[
\cc{\sysalg{\mf a_\bullet}_e}{I} = \cc{e}{I}:\mf a_e.
\]
In particular, for any fixed $c\in R$,  there is some $\psi \in \sysalg{\mf a_\bullet}_e$ such that $\psi(F_*^ec)=1$ if and only if $c\notin \cc{e}{\mf m}:\mf a_e$.
\end{lemma}
\begin{proof}
Since every map $\psi \in \sysalg{\mf a_\bullet}_e$ is of the form $\psi = \Phi^e \star a$ for $a\in \mf a_e$ and $\Phi^e$ the generating map, this means $\psi(F_*^ec) \in I$ for all such $\psi$ if and only if 
\[
(\Phi^e\star a)(F_*^ec) = \Phi^e(F_*^eac) = (\Phi^e\star c)(F_*^ea) \in I
\]
for all $a\in \mf a_e$. In other words, this occurs exactly when
\[
\Phi^e(F_*^e(\mf a_e c)) = (\Phi^e \star c) (F_*^e \mf a_e) \subseteq I
\]
which by the definition of the $e^{\textrm{th}}$ Cartier contraction means $c\in \cc{e}{I}:\mf a_e$ as desired.
\end{proof}

In order to put \Cref{gb-F-graded} to good use, we'll next observe that the idea of using a ``test element'' for strong $F$-regularity also works in the setting of Cartier algebras. 
This is effectively the same argument as in the setting of \cite[Thm.~3.3]{Hochster+Huneke.89a}, 
but we include the proof here for completeness.

\begin{prop}[{C.f.~\cite[Thm.~3.3]{Hochster+Huneke.89a}}]
\label{test-element-pairs}
Let $R$ be a Noetherian $F$-finite ring, and $\ca D$ a Cartier algebra on $R$. Let $s\in R$ be a non-zero divisor such that $(R[s\inv], \ca D[s\inv])$ is strongly $F$-regular. Then $(R,\ca D)$ is strongly $F$-regular if and only if $(R,\ca D)$ is eventually $F$-split along $s$ (i.e., there exists some $e$ and some $\varphi\in \ca D_e$ with $\varphi(F_*^e s) = 1$).
\end{prop}
\begin{proof}
The forward implication is clear by definition.

For the converse, take any non-zerodivisor $c\in R$, and consider the ``evaluation at $c$'' map $\ca D_f \to R$ which has $\phi\mapsto \phi(F_*^f c)$. Since $(R[s\inv], \ca D[s\inv])$ is strongly $F$-regular, when we tensor with $R[s\inv]$ there is some large enough degree $f$ such that this map is surjective in the localization, i.e., there is some $\psi\in \ca D_f$ and some $m$ such that $\psi(F_*^f c) = s^m$. Because $(R,\ca D)$ is eventually $F$-split along $s$, in particular $(R,\ca D)$ is $F$-split, and so there exists $\pi \in \ca D_g$ with $\pi(F_*^g 1)=1$. We can always replace  $m$ by a larger $m$, and replace $g$ by a multiple of $g$, so we reduce to the case where
\[
\psi\in \ca D_f,\ \psi(F_*^f c) = s^{p^g} 
\quad\textrm{and}\quad
\pi \in \ca D_g,\ \pi(F_*^g 1)=1.
\]
Now 
\[
\pi \star \psi(F_*^{f+g} c) = \pi(F_*^g s^{p^g}) = s \pi(F_*^g 1) = s.
\]
Let $\varphi \in \ca D_e$ be our given splitting of $s$. Now finally $\varphi\star\pi\star \psi$ is our desired splitting of $c$.
\end{proof}

In particular, it will be useful for us to have a ready-made collection of elements with which to test strong $F$-regularity, which is what the following corollary provides.

\begin{cor}
\label{radical-test-element}
Let $\mf a_\bullet$ be an $F$-graded system in a strongly $F$-regular $F$-finite local ring $R$.
Let $s$ be a non-zero element of $R$ such that $s\in \sqrt{\mf a_e}$ for all $e\gg 0$. 
Then $\mf a_\bullet$ is strongly $F$-regular if and only if $\mf a_\bullet$ is eventually $F$-split along $s$. In particular, for any non-zero $s\in \mf a_1$, we have that $\mf a_\bullet$ is strongly $F$-regular if and only if $\mf a_\bullet$ is eventually $F$-split along $s$.
\end{cor}
\begin{proof}
Note that ``non-zero divisor'' has become ``non-zero'' since strongly $F$-regular local rings are domains.
Let $\sysalg{\mf a_\bullet}_e = \ca C^R_e \star \mf a_e$ be the Cartier algebra associated to $\mf a_\bullet$.
By the previous fact, it suffices to show that $(R[s\inv], \sysalg{\mf a_\bullet}[s\inv])$ is strongly $F$-regular. Consider any element $c\in R$. Since $R$ is strongly $F$-regular, there exists some $e$ and $\varphi\in \ca C^R_e$ such that $\varphi(F_*^ec) = 1$. Now choose degree $f$ and exponent $a$ such that $\psi\in \ca C^R_f$ is an $F$-splitting and $s^a\in \mf a_{e+f}$, and let $n=\left\lceil \frac{a}{p^{e+f}}\right\rceil$. Then
\[
s^n = s^n(\psi \star \varphi)(F_*^{e+f} c) = \psi \star\varphi\star s^{np^{e+f}} (F_*^{e+f} c).
\]
By design, $s^{np^{e+f}}\in \mf a_{e+f}$ so $\psi\star \varphi \star s^{np^{e+f}} \in \sysalg{\mf a_\bullet}_{e+f}$. Once we localize, this shows $(R[s\inv], \sysalg{\mf a_\bullet})$ is $(e+f)$-$F$-split along $c/1$, and so by a standard result this means that in fact $(R[s\inv],\sysalg{\mf a_\bullet})$ is eventually $F$-split along any element of the form $c/s^t\in R[s\inv]$.

For the ``in particular,'' note that if $s\in \mf a_1$, then $s^{\sum_{i=0}^{e-1}p^i} \in \mf a_e$, and thus $s\in \sqrt{\mf a_e}$ for all~$e$. 
\end{proof}

Using any element from $\mf a_1$ as our test element combined with our understanding of splittings from \Cref{gb-F-graded}, we get a simplified Fedder/Glassbrenner-type criterion for checking $F$-splitting and strong $F$-regularity of an $F$-graded system in this setting.

\begin{cor}
\label{gb-simplified}
Let $(R,\mf m)$ be a Gorenstein and strongly $F$-regular $F$-finite local ring. Let $\mf a_\bullet$ be an $F$-graded system in $R$ with $\mf a_1\neq 0$.
\begin{itemize}
\item $\mf a_\bullet$ is $F$-split if and only if there exists $e>0$ such that $\mf a_e\not\subseteq \cc{e}{\mf m}$.
\item $\mf a_\bullet$ is strongly $F$-regular if and only if there exists $e>0$ such that $\mf a_1 \mf a_e \not\subseteq \cc{e}{\mf m}$.
\end{itemize}
\end{cor}

Now we are ready to prove the main theorem of this section:

\begin{proof}[Proof of \Cref{p-family-f-split-iff-sfr}]
If $\mf b_\bullet$ is strongly $F$-regular, it is by definition also $F$-split. Conversely, suppose $\mf b_\bullet$ is $F$-split. Then by \Cref{gb-simplified}, there exists some $e>0$ and some $c\in \mf b_e\setminus \cc{e}{\mf m}$. 
Since $\mf b_\bullet$ is a $p$-family, $c^{p^f}\in \mf b_e^{[p^f]}\subseteq \mf b_{e+f}$ for all $f\geq 0$. 
By \Cref{radical-test-element}, this $c$ is a test element. 
Further, the ideal $\cc{e}{\mf m}:c$ is proper, and so since $R$ is strongly $F$-regular, 
\[
0=\bigcap_f \cc{f}{\mf m} \supseteq \bigcap_f \cc{f}{\cc{e}{\mf m}:c}.
\]
In particular, there exists some $f$ such that
\[
c\notin \cc{f}{\cc{e}{\mf m}:c} = \cc{e}{\cc{f}{\mf m}}:c^{p^f} = \cc{e+f}{\mf m}:c^{p^f} \supseteq \cc{e+f}{\mf m}:\mf b_{e+f},
\]
where we are making use of the basic properties in \Cref{cc-properties}.
\end{proof}


\section{\texorpdfstring{$p$-Stabilization}{p-Stabilization}}
\label{sec:p-stabilization}


In light of \Cref{p-family-f-split-iff-sfr}, it is natural to consider whether we can use $p$-families to get results for $F$-graded systems more generally. Our goal in this section is thus to present a useful construction which turns $F$-graded systems into $p$-families, in such a way that strong $F$-regularity is preserved.

\begin{defn}
Let $\mf a_\bullet$ be an $F$-graded system. The \emph{$p$-stabilization} of $\mf a_\bullet$ is $\pstab {\mf a}_\bullet$, where
\[
\pstab {\mf a}_e := \left\{ r \: \big| \: r^{p^{f}} \in \mf a_{f+e} \textrm{ for all $f\gg 0$}\right\}.
\]
\end{defn}

From this definition, we immediately get the following result:
\begin{fact}
\label{p-stab-is-p-family}
The $p$-stabilization of any $F$-graded system is a $p$-family.
\end{fact}
\begin{proof}
Once can check $\pstab{\mf a}_e$ is an ideal. For the $p$-family condition, if $r\in \pstab{\mf a}_e$, then for all $f\gg 0$, we have $(r^p)^{p^f} = r^{p^{f+1}} \in \mf a_{(e+1)+f}$.
\end{proof}

Further, strong $F$-regularity is indeed preserved:

\begin{thm}
\label{a-SFR-iff-b-SFR}
Let $(R,\mf m)$ be a Gorenstein and strongly $F$-regular $F$-finite local ring. 
Let \(\mf a_\bullet\) be an $F$-graded system in $R$ with $\mf a_1\neq 0$, and let $\pstab{\mf a}_\bullet$ be the $p$-stabilization of $\mf a_\bullet$. 
Then $\mf a_\bullet$ is strongly $F$-regular if and only if $\pstab{\mf a}_\bullet$ is strongly $F$-regular. 
\end{thm}

Combining this theorem and \Cref{p-family-f-split-iff-sfr} with the fact that $\pstab{\mf a}_\bullet$ is a $p$-family, the following corollary is immediate.
\begin{cor}
\label{a-SFR-iff-b-Fsplit}
Let $(R,\mf m)$ be a Gorenstein and strongly $F$-regular $F$-finite local ring. 
Let \(\mf a_\bullet\) be an $F$-graded system in $R$ with $\mf a_1\neq 0$, and let $\pstab{\mf a}_\bullet$ be the $p$-stabilization of $\mf a_\bullet$. 
Then $\mf a_\bullet$ is strongly $F$-regular if and only if $\pstab{\mf a}_\bullet$ is $F$-split. 
\end{cor}

In order to prove \Cref{a-SFR-iff-b-SFR}, the main theorem of this section, we will first need a more in-depth study of $p$-stabilization. Thus the actual proof is postponed until \cpageref{pf:a-SFR-iff-b-SFR}. To start, the following straightforward observation will be useful for examples. 

\begin{lemma}
If $R$ is $\bb N^d$ graded and $\mf a_\bullet$ is a homogeneous $F$-graded system, then so is $\pstab{\mf a}_\bullet$. In particular, if $R$ is a polynomial ring and $\mf a_\bullet$ is a monomial $F$-graded system, then $\pstab{\mf a}_\bullet$ is also monomial.
\end{lemma}
\begin{proof}
Write $r= \sum_\alpha r_\alpha\in R$ as a decomposition into graded pieces.
If $r^{p^f}=\sum_\alpha r_\alpha^{p^f}\in \mf a_{e+f}$, then the degree $p^f\alpha$ piece of $r^{p^f}$ (which is $r_\alpha^{p^f}$) is also in $\mf a_{e+f}$ as desired.
\end{proof}

Monomial $F$-graded systems will be analyzed more in \Cref{sec:shapes}. Now we proceed to some further general properties:

\begin{prop}[Basic properties of $p$-stabilization]
\label{p-stabilization-basic-properties}
Let $R$ be a Noetherian $F$-finite ring, and let $\mf a_\bullet$ be an $F$-graded system of $R$.
\begin{enumerate}
\item $\pstab{\pstab{\mf a}}_e = \pstab{\mf a}_e$ for all~$e$, i.e., any $p$-stabilized system is itself \emph{$p$-stable}.

\item If further $\mf b_\bullet$ is an $F$-graded system with $\mf a_e \subseteq \mf b_e$ for all $e\gg 0$, then $\pstab{\mf a}_e \subseteq \pstab{\mf b}_e$ for all~$e$.

\item If $\mf a_\bullet$ is a $p$-family, then $\mf a_e \subseteq \pstab{\mf a}_e$ for all $e$.
\end{enumerate}
Thus $p$-stabilization behaves like a ``closure'' operation on $p$-families.
\end{prop}
\begin{proof}
We will prove these out of order, starting with the second and third:
\begin{enumerate}[start=2]
\item Let $r\in \pstab{\mf a}_e$, so that $r^{p^f}\in \mf a_{e+f}$ for all $f\gg 0$. But by taking sufficiently large $f$, we get $r^{p^f}\in \mf a_{e+f} \subseteq \mf b_{e+f}$, and thus $r\in \pstab{\mf b}_e$.

\item If $r\in \mf a_e$, then since it is a $p$-family, of course $r^{p^f} \in \mf a_{e+f}$ for all~$f$, and thus $r\in \pstab{\mf a}_e$. 
\end{enumerate}
Now we return to the first property on the list:
\begin{enumerate}
\item By using the other two properties just proven and the fact that $\pstab{\mf a}_\bullet$ is a $p$-family, we automatically get $\pstab{\mf a}_e \subseteq \pstab{\pstab{\mf a}}_e$ for all~$e$. 
Conversely, take $r\in \pstab{\pstab{\mf a}}_e$. 
Then for all $f\gg 0$ and $g\gg 0$, $r^{p^{f+g}} = (r^{p^f})^{p^g} \in \mf a_{(e+f)+g}$ as desired.\qedhere
\end{enumerate}
\end{proof}

As a caution to the reader, we note that the restriction to $p$-families in the third property of \Cref{p-stabilization-basic-properties} is necessary: in \Cref{colon-system-stabilization}, we will see an example of an $F$-graded system where in fact $\pstab{\mf a}_e \subseteq \mf a_e$ for all~$e$, namely the system where $\mf a_e = I^{[p^e]}:I$ for some fixed ideal $I$, yielding the smaller stabilization $\pstab{\mf a}_e = I^{[p^e]}$. 
We also note that the containment in the third property can indeed sometimes be strict:
\begin{zb}
Fix a constant $c\in \N$ with $c\geq 1$, and consider the $p$-family $\mf a_e = \langle x^{c+p^e}\rangle \subset k[x]$. 
However, $\pstab{\mf a}_e = \langle x^{1+p^e}\rangle$, since monomial $x^a\in \pstab{\mf a}_e$ means $(x^a)^{p^f} \in \langle x^{c+p^{e+f}}\rangle$ for all $f\gg 0$, i.e., $a \geq \frac{c}{p^f}+p^e$ for all $f\gg 0$. In particular, if $c> 1$, the $p$-stabilization is strictly larger than the original family.
\end{zb}
In fact, even exponential growth is not enough if the growth factor is less than $p$:
\begin{zb}
\label{zb:exponential-growth-insufficient}
Fix a constant $c\in \N$ and consider the system $\mf a_e = \langle x^{c^e +  p^e}\rangle \subset k[x]$. 
This is a $p$-family if and only if $c^{e+1}\leq p\cdot c^e $ for all~$e$. 
One can further check that a monomial $x^a \in \pstab{\mf a_e}$ if and only if $a\geq \frac{c^{e+f}}{p^f} + p^e$ for all~$f\gg 0$.
In particular, if $0<c<p$, then $\pstab{\mf a}_e = \langle x^{1+p^e}\rangle$.
\end{zb}

In light of these examples and the first property of $p$-stabilization (\Cref{p-stabilization-basic-properties}), it would be interesting to identify the $F$-graded systems that are \emph{$p$-stable}, i.e., the systems $\mf a_\bullet$ such that $\pstab{\mf a}_e = \mf a_e$ for all~$e$. 
To be $p$-stable, it is clearly necessary for the system to be a $p$-family to begin with. We now work towards giving a sufficient condition.

\begin{lemma}
\label{stable-inside-map-f-graded}
Let $R$ be a Noetherian $F$-finite ring, and let $\mf a_\bullet$ be an $F$-graded system. Then for any fixed  surjective degree $d$ map $\varphi \in \ca C^R_d$, we have 
\[
\pstab{\mf a}_e \subseteq \bigcup_{m>0} \bigcap_{n\geq m} \varphi^{\star n}(F_*^{nd}\mf a_{e+nd}).
\]
\end{lemma}
\begin{proof}
Since $\varphi$ is surjective, in particular there is some $s\in R$ with $\varphi(F_*^d s) =1$, so that further $\varphi^{\star n}(F_*^{nd} s^{1+p^d+\cdots + p^{(n-1)d}}) = 1$ for all $n$. Then consider any $r\in \pstab{\mf a}_e$, so that $r^{p^f}\in \mf a_{e+f}$ for all $f\gg 0$. In particular, for $n\gg 0$ we have
\[
r = r\cdot \varphi^{\star n}(F_*^{nd} s^{(p^{nd}-1)/(p^d-1)}) 
= \varphi^{\star n} (F_*^{nd}(r^{p^{nd}}s^{(p^{nd}-1)/(p^d-1)}))\in \varphi^{nd}(F_*^{nd}\mf a_{e+nd}).
\]
This gives the desired containment.
\end{proof}

\begin{prop}
\label{map-gives-stable-p-family}
Let $R$ be a Noetherian $F$-finite ring, 
and let $\mf b_\bullet$ be a $p$-family in $R$. 
Suppose there exists some $d$ and some map $\varphi\in \ca C^R_d$ such that $\varphi$ is surjective and for all $e$, we have $\varphi(F_*^d \mf b_{e+d}) \subseteq \mf b_e$. 
Then $\mf b_\bullet$ is \emph{$p$-stable}, i.e., $\pstab{\mf b}_e = \mf b_e$ for all $e$. 
\end{prop}
\begin{proof}
Note that the condition iterates in the sense that 
\[
\varphi^{\star n}(F_*^{nd}\mf b_{e+nd}) \subseteq \varphi^{\star (n-1)}(F_*^{(n-1)d} \mf b_{e+(n-1)d}) \subseteq \cdots \subseteq \mf b_e.
\]
Thus
\[
\pstab{\mf b}_e \subseteq \bigcup_{m>0}\bigcap_{n\geq m} \varphi^{\star n}(F_*^{nd}\mf a_{e+nd}) 
\subseteq \bigcup_{m>0}\bigcap_{n\geq m} \mf b_e = \mf b_e.
\]
Since $\mf b_\bullet$ is already a $p$-family, $\mf b_e\subseteq \pstab{\mf b}_e$ by the third property of $p$-stabilization (\Cref{p-stabilization-basic-properties}) and so $\mf b_\bullet$ is stable, as desired. 
\end{proof}

\begin{rmk}
The requirement that $\mf b_\bullet$ be a $p$-family in the previous proposition is indeed necessary. Consider the $F$-graded system $\mf b_e = \langle x^{1+p+\cdots + p^{e-1}}\rangle$ in $k[x]$, and let $\varphi$ be the standard monomial splitting in $\hom_{k[x]}(F_*k[x], k[x])$, which sends $F_*1$ to $1$. Then
\[
\varphi\left(F_* \mf b_{e+1}\right) = \varphi\left(F_* \langle x^{1+p+\cdots + p^{e-1}}\rangle\right)
    = x^{1+\cdots + p^{e-2}} \varphi\left(F_*\langle x\rangle\right)
    = \mf b_{e-1} \varphi(F_*\langle x \rangle ) = x\mf b_{e-1},
\]
so we have the desired containment.
But the system is not $p$-stable---by \Cref{minl-system-stabilization}, $\pstab{\mf b}_e=x\mf b_{e-1}$.
\end{rmk}

\begin{zb}
\label{zb:cc-p-stable}
Let $(R,\mf m)$ be an $F$-split local ring, and let $I$ be an ideal. Then $\cc{\bullet}{I}$ is $p$-stable. To see this, let $\varphi\in \ca C^R_1$ be a Frobenius splitting, and consider $r\in \cc{e+1}{I}$. Then for all $\psi \in \ca C^R_e$, we have
\[
\psi\star\varphi(F_*^{e+1}r) = \psi(F_*^e \varphi(F_*r)) \in I,
\]
since $\psi \star \varphi \in \ca C^R_{e+1}$. Thus $\varphi(F_*r)\in A_e(I)$, and because the system is already a $p$-family, \Cref{map-gives-stable-p-family} lets us conclude that $\cc{\bullet}{I}$ is indeed $p$-stable.
\end{zb}

These results prompt the following conjecture:

\begin{conj}
\label{conj:stable-p-family}
Let $R$ be a strongly $F$-regular ring, 
and let $\mf b_\bullet$ be a $p$-family in $R$. Then $\mf b_\bullet$ is \emph{$p$-stable}, i.e., $\pstab{\mf b}_e = \mf b_e$ for all $e$, if and only if there exists some $d$ and some map $\varphi\in \ca C^R_d$ such that $\varphi$ is surjective and for all $e$, we have $\varphi(F_*^d \mf b_{e+d}) \subseteq \mf b_e$. 
\end{conj}

This conjecture holds when $\mf b_\bullet$ is a system of monomial ideals, as we shall see in \Cref{stable-monomial-families}.

\begin{rmk}
A non-strongly $F$-regular ring, and even a non-$F$-split ring can certainly have stable $p$-families. For example, the $p$-family $\mf b_e = \langle 1 \rangle$ for all $e$ is always $p$-stable, but in a non-$F$-split ring there are no surjective maps, so our conjecture fails. However, we can't just weaken the hypothesis to being any map. For example, in $S=k[x]$, the family $\mf b_e = \langle x^{p^e+2}\rangle$ is not stable, because 
\[
x^{p^2+2} \cdot x^{p-2} = (x^{p+1})^p \in \mf b_2
\]
but $x^{p+1}\notin \mf b_1$. 
However, taking $\varphi = x^2 \star \Phi$, where $\Phi$ generates $\hom_S(F_*S,S)$, we see that for any $rx^{p^{e+1}+2} \in \mf b_{e+1}$, 
\[
\varphi(F_*(rx^{p^{e+1}+2})) = x^{p^e}\varphi(F_*(rx^2)) = x^{p^e+2}\Phi(r) \in \mf b_e.
\]
\end{rmk}

As our final lemma towards relating $\mf a_\bullet$ and $\pstab{\mf a}_\bullet$, it will help to have some constraints on how far apart $\mf a_e$ and $\pstab{\mf a}_e$ can get:

\begin{lemma}
\label{ag-ae-subset-be}
Let $\mf a_\bullet$ be an $F$-graded system. Then for every $e$, $\mf a_1 \cdot \mf a_e \subseteq \pstab{\mf a}_e$.
\end{lemma}
\begin{proof}
For any $f\in \N$, we have
\[
\mf a_{e+f} \supset \mf a_e^{[p^{f}]}\mf a_{f} \supset \mf a_e^{[p^{f}]}\cdot \prod_{i=0}^{f-1} \mf a_1^{[p^{i}]}.
\]
In particular, this means that for $r\in \mf a_1$, $s\in \mf a_e$, we have
\[
\mf a_{e+f} \ni s^{p^{f}}   r^{\sum_{i=0}^{f-1} p^{i}} = s^{p^{f}} r^{\frac{p^{f}-1}{p -1}}.
\]
Thus $(sr)^{p^{f}}\in \mf a_{e+f}$ for all $f\geq 1$, which means that $rs \in \pstab{\mf a}_e$ as desired.
\end{proof}
In the case that $\mf a_1=0$, the previous lemma is of course not very informative.

\begin{rmk}
One might hope to have a containment showing ``something related to $\pstab{\mf a}_e$'' is contained in ``something related to $\mf a_e$''. There are unfortunately some limitations here.  
Taking $\mf a_e = \langle x^{p^e+c^e}\rangle $ and $\pstab{\mf a}_e = \langle x^{p^e+1}\rangle$ for integer $0<c<p$ as in \Cref{zb:exponential-growth-insufficient} shows that for any monomial ideal $I=\langle x^a\rangle$, we have that eventually, $I \cdot \pstab{\mf a}_e \not\subseteq \mf a_e$. So there is no fixed multiplicative factor that will work in this direction.
\end{rmk}

We are finally ready to prove the main result of this section.

\begin{proof}[Proof of \Cref{a-SFR-iff-b-SFR}]
\label{pf:a-SFR-iff-b-SFR}
To show the contrapositive, assume that $\pstab{\mf a}_\bullet$ is not strongly $F$-regular. By \Cref{p-family-f-split-iff-sfr}, this means $\pstab{\mf a}_\bullet$ is also not $F$-split, and by \Cref{gb-simplified} this tells us that for all $e$, $\pstab{\mf a}_e\subseteq \cc{e}{\mf m}$. But then \Cref{ag-ae-subset-be} further tells us
\[
\mf a_1 \mf a_e \subseteq \pstab{\mf a}_e \subseteq \cc{e}{\mf m}
\]
for all $e$, and so $\mf a_e$ cannot be strongly $F$-regular.

For the other  direction, now assume that $\pstab{\mf a}_\bullet$ is strongly $F$-regular. Thus there exists an $e$ such that $\pstab{\mf a}_1\pstab{\mf a}_e \not\subseteq \cc{e}{\mf m}$, i.e., there is $r\in \pstab{\mf a}_1$, $s\in \pstab{\mf a}_e$ with $rs \notin \cc{e}{\mf m}$. 
On the one hand, \Cref{zb:cc-p-stable} ensures $(rs)^{p^{e'}}\notin \cc{e+e'}{\mf m}$ for every $e'$, since $A_\bullet(\mf m)$ is $p$-stable. 
On the other hand, by definition of $\pstab{\mf a}_\bullet$, for all $e'\gg 0$ we have $r^{p^{e'}}\in \mf a_{1+e'}$ and $s^{p^{e'}} \in \mf a_{e+e'}$. 
Let $e_0$ be an $e'\gg 0$ which satisfies both conditions, i.e., $r^{p^{e_0}}\in \mf a_{1+e_0}$ and $s^{p^{e_0}}\in \mf a_{e+e_0}$.

Then $(rs)^{p^{e_0}}\in \mf a_{1+e_0}\mf a_{e+e_0} \setminus \cc{e+e_0}{\mf m}$, which in particular means $r^{p^{e_0}}\notin \cc{e+e_0}{\mf m}:\mf a_{e+e_0}$. Since $r\in \pstab{\mf a}_1$, this means $r$, and also $r^{p^{e_0}}$, is in $\sqrt{\mf a_{e}}$ for all $e\gg 0$, and thus $r^{p^{e_0}}$ is a test element in the sense of \Cref{radical-test-element} such that $\mf a_\bullet$ is eventually $F$-split along $r^{p^{e_0}}$, as desired.
\end{proof}

\subsection{Application: Our Three Main Examples}
\label{sec:stabilization-zb}

Now that we've seen some benefits of $p$-stabilization, we will show how to compute the $p$-stabilization for (some special cases of) our main examples from \Cref{zb:basic-three}. A key tool for computing the $p$-stabilization of a monomial idea will be the associated $p$-body, introduced in \Cref{sec:shapes}. In fact, many of the proofs here will be deferred to \Cref{sec:shapes-zb}, after we have developed the necessary machinery.

\begin{thm}
\label{minl-system-stabilization}
Let $S=k[x_1,\ldots, x_d]$ for $k$ an $F$-finite field, and let $I$ be a monomial ideal with minimal monomial generating set $\{x^\nu \: | \: \nu\in \mc V\}$ where $\mc V\subset \bb N^d$. Let $\mf a_e = \prod_{i=0}^{e-1} I^{[p^i]}$. 
Define
\[
J = \left\langle x^{\lceil \sum_\nu c_\nu \nu\rceil} \: \Big| \: c_\nu \in \R_{\geq 0}, \ \sum_{\nu\in \mc V} c_\nu = \frac{1}{p-1} \right\rangle.
\]
Then the $p$-stabilization is
\[
	\pstab{\mf a}_e = J\cdot \mf a_e = J \cdot\prod_{i=0}^{e-1} I^{ [p^i] }.
\]
\end{thm}
This theorem will be proven on page~\pageref{proof-minl-system-stabilization}.

\begin{thm}
\label{colon-system-stabilization}
Let $R$ be an $F$-finite regular local ring, fix a non-zero ideal $I$ in $R$, and define $\mf a_e = I^{[p^e]}:I$. 
Then $\pstab{\mf a}_e = I^{[p^e]}$. 
\end{thm}
\begin{proof}
We see first that
\[
(I^{[p^e]})^{[p^f]} = I^{[p^{e+f}]} \subseteq I^{[p^{e+f}]}:I,
\]
and so $I^{[p^e]}\subseteq \pstab{\mf a}_e$. On the other hand, if $r^{p^f}\in \mf a_{e+f}=I^{[p^{e+f}]}:I$, then
\[
I\subseteq I^{[p^{e+f}]}:r^{p^f} = (I^{[p^e]}:r)^{[p^f]}.
\]
But this is supposed to hold for all sufficiently large $f$. If $I^{[p^e]}:r$ is a proper ideal, then Krull's intersection theorem would tell us that $I=0$. Otherwise, this means that $r\in I^{[p^e]}$ as desired.
\end{proof}

\begin{thm}
\label{maxl-t-system-stabilization}
Let $S = k[x_1,\ldots, x_d]$ for $k$ an $F$-finite field, let $\mf m$ be the homogeneous maximal ideal, and fix $t\in \R_{\geq 0}$. Let $\mf a_e = \mf m^{\lceil t(p^e-1)\rceil}$. Then the $p$-stabilization is
\[
	\pstab{\mf a}_e 
= \mf m^{\lceil tp^e\rceil}.
\]
\end{thm}
This theorem will be proven on page~\pageref{proof-maxl-t-system-stabilization}.

\section{\texorpdfstring{The Associated $p$-Body}{The Associated p-Body}}
\label{sec:shapes}
We now pivot to \emph{monomial $F$-graded systems}, i.e., $F$-graded systems in a polynomial ring for which every ideal is a monomial ideal. 
In this setting we develop a technique for computing the $p$-stabilization, which comes by way of a geometric construction (the \emph{associated $p$-body}) that is interesting in its own right. In this section, we assume the ambient ring $S=k[x_1,\ldots, x_d]$ is a polynomial ring over an $F$-finite field $k$.

\begin{notation}
If $I$ is a monomial ideal, then $\log I = \{\alpha \in \N^d \: | \: x^\alpha \in I\}$ is the set of exponent vectors.
For $n\in \N$, we write $[n]:=\{1, 2, \ldots, n\}$.
Whenever possible, we will use lowercase Greek letters to denote vectors (be they in $\N^d$, $\left(\Z[1/p]\right)^d$, or $\R^d$). 
We use the termwise partial order $\leq$ on vectors: given two vectors $\alpha, \beta$, write $\alpha \leq \beta$ if $\alpha_i \leq \beta_i$ for all coordinates~$i$. The norm of a vector is $|\alpha| = \sum_i \alpha_i$. Finally, we use $\mbbm 1$ to denote the vector which has every coordinate equal to~$1$.
\end{notation}

Given a monomial ideal in $k[x_1,\ldots, x_d]$, it is often the case that looking at corresponding diagrams of exponent vectors in $\N^d$ can shed light on the algebraic picture. A version of this correspondence is also illuminating for $F$-graded systems:

\begin{defn}
\label{def:associated-shape}
Let $\mf a_\bullet$ be a monomial $F$-graded system in $k[x_1,\ldots, x_d]$. Then the \emph{associated $p$-body} in $\left(\N\left[1/p\right]\right)^d$ is
\[
\Delta(\mf a_\bullet) := \bigcup_{f>0}\  \bigcap_{e\geq f} \frac{1}{p^e}\log \mf a_e 
= \left\{\alpha \in \left(\N\left[\frac{1}{p}\right]\right)^d\: \bigg| \: p^e\alpha \in \log \mf a_e \ \forall e \gg 0 \right\}.
\]
\end{defn}

\begin{prop}
\label{associated-p-body-invt}
Let $\mf a_\bullet$ be a monomial $F$-graded system, and let 
\(
\Delta = \Delta(\mf a_\bullet).
\)
Then 
\[
\Delta = \Delta + \left(\N\left[\frac{1}{p}\right]\right)^d.
\]
\end{prop}
\begin{proof}
Since $\ulin 0\in \left(\N\left[\frac{1}{p}\right]\right)^d$, the containment $\subset$ is clear. 

Take $\alpha \in \Delta$ and let $\beta \in \left(\N\left[\frac{1}{p}\right]\right)^d$. We will show that $\alpha + \beta \in \Delta$. Choose $G\in \N$ such that $p^G\beta \in \N^d$. By definition of $\Delta$, there exists $g\geq G$ such that for all $e\geq g$, we have $x^{p^e\alpha} \in \mf a_e$. Since our choice of $G$ ensures $x^{p^e\beta} \in S$, this means that $x^{p^e\alpha}x^{p^e\beta}\in \mf a_e$, so that $\alpha + \beta \in \frac{1}{p^e} \log \mf a_e$ for all $e\gg 0$, as desired. 
\end{proof}

\begin{rmk}
	As we will be primarily interested in using $p$-bodies in the context of the $p$-body/$p$-stabilization correspondence (see \Cref{shape-stabilization-correspondence} below), $\Delta(\mf a_\bullet)$ is functionally equivalent to the subset of $\R^d$ obtained by taking $\Delta(\mf a_\bullet) + \R_{\geq 0}^d$, since the $(\Z[1/p])^d$ points are the same.

    This perspective also explains the source of the name. Working in the setting of a local ring accompanied by a sufficiently nice valuation, Hern\'andez and Jeffries introduce a subset of~$\R^d$ called the \emph{associated $p$-body} to a collection of subsets in a semigroup called a \emph{$p$-system} \cite[Def.~4.4]{Hernandez+Jeffries.18}. In Section~5 of their paper they then apply this to $p$-families to get the associated $p$-body living in~$\R^d$.
\end{rmk}

\begin{defn}
\label{def:associated-p-family}
Let $\Delta$ be any subset of $\left(\N[\frac{1}{p}]\right)^d$. 
Then we can define an \emph{associated $p$-family} $\mf a^\Delta_\bullet$ in $k[x_1,\ldots, x_d]$, where
\[
\mf a^\Delta_e := \left\langle x^{\beta} \: \Big| \: \frac{1}{p^e}\beta \in \Delta\right\rangle.
\]
\end{defn}

This is a $p$-family because for any $x^{\beta} \in \mf a^\Delta_e$, we have $\frac{p\beta}{p^{e+1}} = \frac{1}{p^e}\beta\in \Delta$ and so $x^{p\beta} \in \mf a^\Delta_{e+1}$. 

\begin{thm}[$p$-body/$p$-stabilization correspondence]
\label{shape-stabilization-correspondence}
If $\mf b_\bullet$ is a monomial $F$-graded system, then the associated $p$-family of the associated $p$-body of $\mf b_\bullet$ is the $p$-stabilization, i.e., $\mf a_\bullet^{\Delta(\mf b_\bullet)} = \pstab{\mf b}_\bullet$. 
If $\Delta\subseteq (\N[1/p])^d$,
then
$\Delta(\mf a^\Delta_\bullet) = \Delta+ \left(\N[1/p]\right)^d$.

In particular, this gives a correspondence between monomial $p$-stable $F$-graded systems and subsets of $(\N[1/p])^d$ which are invariant under adding $\left( \N[1/p]\right)^d$. 
\end{thm}
\begin{proof}
A straightforward computation shows
\begin{align*}
\mf a^{\Delta(\mf b_\bullet)}_e 
    &= \left\langle x^{\beta} \: \Big| \: \frac{1}{p^e}\beta \in \Delta(\mf b_\bullet)\right\rangle
    = \left\langle x^{\beta} \: \Big| \: \frac{p^f}{p^e}\beta \in \log(\mf b_f) \ \forall f\gg 0\right\rangle \\
    &= \left\langle x^{\beta} \: \Big| \: x^{p^{f-e}\beta} \in \mf b_f \ \forall f\gg 0\right\rangle 
    = \pstab{\mf b}_e.
\end{align*}

For the statement about $\Delta$, 
since $\mf a^\Delta_\bullet$ is a $p$-family, we have $\frac{1}{p^e}\log \mf a^{\Delta}_e \subseteq \frac{1}{p^{e+1}}\log \mf a^{\Delta}_{e+1}$. Thus
\[
\Delta(\mf a^\Delta_\bullet) = \bigcup_{f>0}\ \bigcap_{e\geq f} \frac{1}{p^e} \log(\mf a^\Delta_e) 
    = \bigcup_{f>0} \frac{1}{p^f} \log(\mf a^\Delta_f).
\]

The generating monomials of $\mf a_f^\Delta$ come from the lattice points in $\Delta \cap \left(\frac{1}{p^f}\N^d\right)$. Since $\mf a_f^\Delta$ is an ideal, the monomials correspond to the lattice points in $\left(\Delta \cap \left(\frac{1}{p^f}\N^d\right)\right)+ \frac{1}{p^f}\N^d$, i.e., we have 
\[
\frac{1}{p^f}\log(\mf a^\Delta_f) = \left(\Delta \cap \left(\frac{1}{p^f}\N^d\right)\right) + \frac{1}{p^f} \N^d,
\]
and so taking the union over all $f>0$ gives our desired result.
\end{proof}

Now we will see our first example of using an associated $p$-body to find the $p$-stabilization:
\begin{zb}
	\label{one-variable-p-stable}
	If $\Delta \subseteq \N[1/p]$ is invariant under adding $\N[1/p]$, then there exists some $a \in \R$ such that either $\Delta = [a,\infty)\subset \N[1/p]$ or $\Delta = (a,\infty)\subset \N[1/p]$. 
    In particular, any $p$-stable system $\mf b_\bullet$ looks like either 
    $\mf b_e = \langle x^{\lceil ap^e\rceil}\rangle$ for $e>0$ or $\mf b_e=\langle x^{\lfloor ap^e\rfloor+1}\rangle$ for $e> 0$.
\end{zb}

This example also prompts the return of the question: which (monomial) $F$-graded systems are $p$-stable? The $p$-body/$p$-stabilization correspondence (\Cref{shape-stabilization-correspondence}) gives a geometric answer to this question. 
But there is also a description in terms of a distinguished map, in support of \Cref{conj:stable-p-family}:  
\begin{prop}
	\label{stable-monomial-families}
	Let $\mf b_\bullet$ be a monomial $p$-family. Then $\mf a_\bullet^{\Delta(\mf b_\bullet)} = \mf b_\bullet$ if and only if for all~$e$, $\varphi(F_*\mf b_{e+1}) \subseteq \mf b_e$, 
	where $\varphi\in \hom_S(F_*S, S)$ is the standard monomial splitting which sends $F_*1\mapsto 1$ and the other standard monomial generators of $F_*S$ to $0$.
\end{prop}
\begin{proof}
	If $\varphi(F_*\mf b_{e+1})\subseteq \mf b_e$ for all $e$, then \Cref{map-gives-stable-p-family} ensures that $\pstab{\mf b}_\bullet = \mf b_\bullet$.

	Conversely, suppose that $\mf a_\bullet^{\Delta(\mf b_\bullet)} = \mf b_\bullet$. To show $\varphi(F_*\mf b_{e+1}) \subseteq \mf b_e$, it suffices to consider what happens on monomials. 
	So, consider $x^\alpha \in \mf b_{e+1}$. 
	By assumption, $\alpha/p^{e+1}\in \Delta$, which means for all $f\gg 0$ we have $\frac{\alpha}{p^{e+1}} \in \frac{1}{p^{e+1+f}} \log \mf b_{e+1+f}$. 
	We also know that for a general monomial, the standard monomial splitting sends
	\[
		\varphi(F_* x^\alpha) = 
		\begin{cases}
			x^{(n-1)\mbbm 1} & \alpha = np\mbbm 1 \textrm{ for some positive integer $n$}\\
			0 & \textrm{else}.
		\end{cases}
	\]
	In the latter case the result is clearly in $\mf b_e$, so suppose $\alpha = np\mbbm 1$. But then $\frac{1}{p^e}\alpha = \frac{n}{p^e}\mbbm 1  \in \Delta$, and so $x^{n\mbbm 1} \in \mf a_e^{\Delta(\mf b_\bullet)} = \mf b_e$ as desired.
\end{proof}

\subsection{Application: Our Three Main Examples}
\label{sec:shapes-zb}

Again, now that we have seen some benefits of the associated $p$-body, we will show how to compute it for (some special cases of) our main examples. This will also allow us to describe the $p$-stabilizations of the corresponding ideals that were promised in \Cref{sec:stabilization-zb}.

\begin{thm}
	\label{minl-system-shape}
	Let $I$ be a monomial ideal in $S$. Suppose that $I$ has minimal monomial generating set $\{x^\nu \: | \: \nu\in \mc V\}$ where $\mc V\subset \bb N^d$, and let $\mf a_e = \prod_{i=0}^{e-1} I^{[p^i]}$. Consider the set of vectors
	\[
		\mc W = \left\{
			\sum_{i=1}^\infty v(i) \frac{1}{p^i} \: \bigg| \: v: \N_{>0} \to \mc V
			\right\} \subset \R^d.
	\]
	Then
	\[
		\Delta(\mf a_\bullet) = \left\{
			\mu \in \left( \Z[1/p]\right)^d \: \big| \: \exists\, \omega \in \mc W \textrm{ s.t. } \mu \geq \omega
			\right\}.
	\]
	In other words, $\Delta(\mf a_\bullet)$ is the $(\N[1/p])^d$-invariant subset generated by $\mc W$.
\end{thm}
An example of the process of computing this associated $p$-body is shown in \Cref{fig:minl-sys-x3-y6} with starting ideal $I=\langle x^3,y^6\rangle$ in characteristic $3$. 

\begin{figure}
\centering
\begin{subfigure}{0.435\textwidth}
\centering
\includegraphics[width=0.7\textwidth, trim={15 50 15 40}, clip]{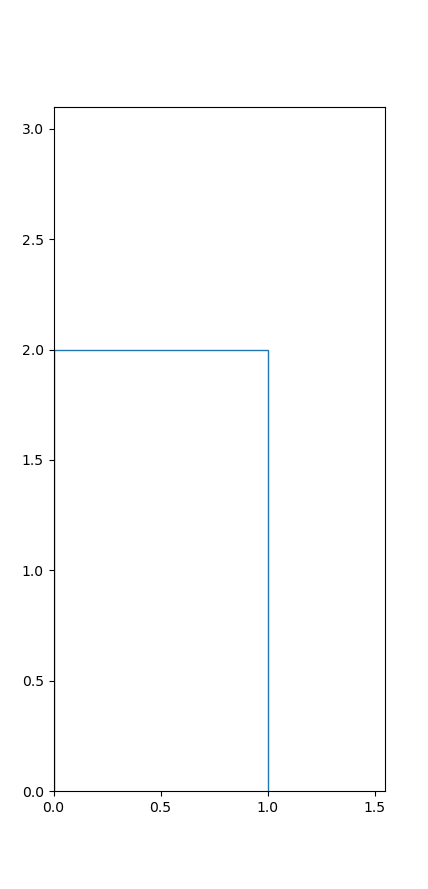}
\caption{$e=1$}
\end{subfigure}\hfill%
\begin{subfigure}{0.435\textwidth}
\centering
\includegraphics[width=0.7\textwidth,trim={15 50 15 40}, clip]{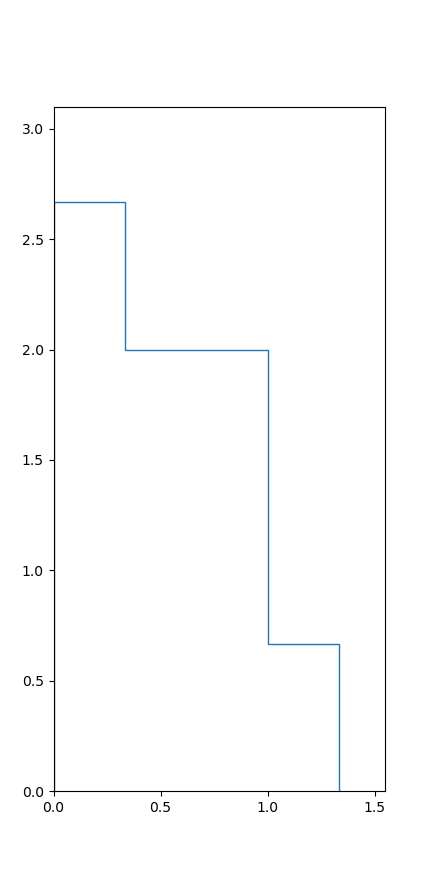}
\caption{$e=2$}
\end{subfigure}
\begin{subfigure}{0.435\textwidth}
\centering
\includegraphics[width=0.7\textwidth,trim={15 50 15 40}, clip]{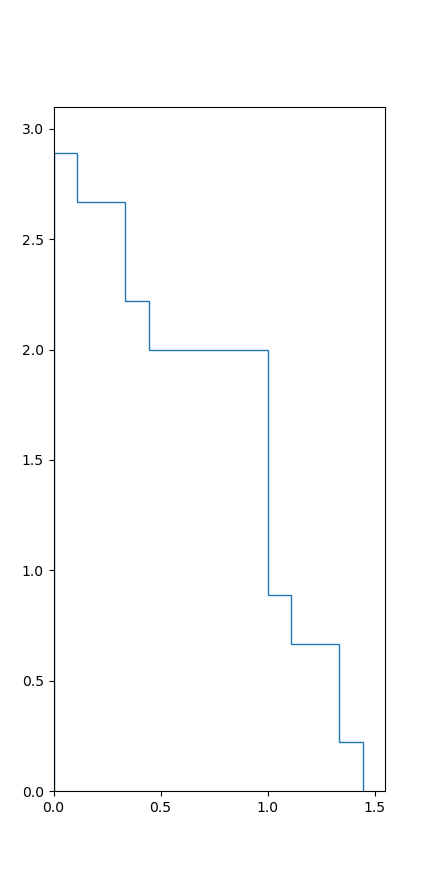}
\caption{$e=3$}
\end{subfigure}\hfill%
\begin{subfigure}{0.435\textwidth}
\centering
\includegraphics[width=0.7\textwidth,trim={15 50 15 40}, clip]{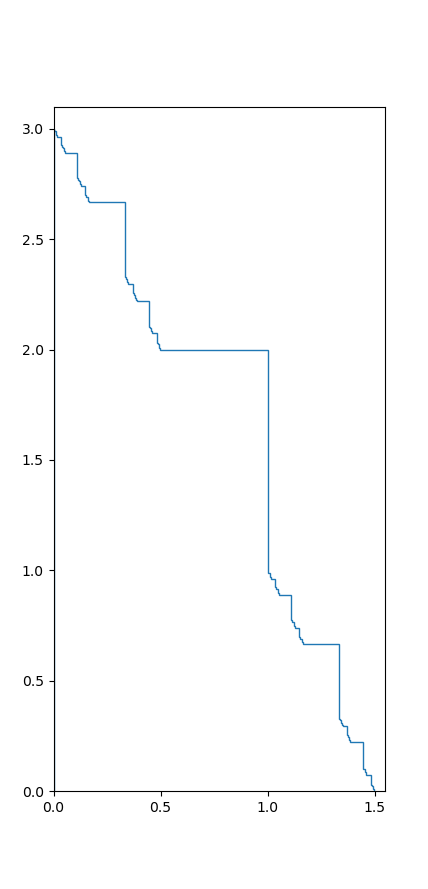}
\caption{$e=8$}
\end{subfigure}%
\caption{Plots of $\frac{1}{p^e}\log \mf a_e$ when $I=\langle x^3,y^6\rangle$, $\mf a_e = \prod_{i=0}^{e-1}I^{[p^i]}$, and $p=3$. The relevant lattice points of $\frac{1}{p^e}\N$ lie above and to the right of the blue line in each subfigure.}
\label{fig:minl-sys-x3-y6}
\end{figure}

\begin{proof}
	Note that each sum in $\mc W$ actually converges, since there are only finitely many vectors in $\mc V$.
	First, suppose that $\mu\in \Z[1/p]^d$ and that there exists some $\omega\in \mc W$ with corresponding function $v:\N_{>0}\to \mc V$ such that $\mu\geq \omega$. Consider any $e\in \N$ such that $p^e \mu\in \Z^d$. By assumption, 
	\[
		p^e \mu \geq p^e \omega \geq p^e \left(\sum_{i=1}^e v(i)\frac{1}{p^i}\right) = \sum_{i=0}^{e-1} v(e-i) p^i.
	\]
	In particular, this means that $x^{p^e\mu}$ is a multiple of $x^{v(e)}(x^{v(e-1)})^p\cdots (x^{v(1)})^{p^{e-1}} \in \prod_{i=0}^{e-1}I^{[p^i]}$.

	For the other direction, suppose $\mu\in \Delta$. 
    This means for all $e\gg 0$, $\mu\in \frac{1}{p^e}\log \mf a_e$. More specifically, for all such $e$, there exists $v^{(e)}:[e]\to \mc V$ such that $\sum_{i=1}^e v^{(e)}(i)\frac{1}{p^i}\leq \mu$. Since $\mc V$ is a finite set, we can iteratively define a function $\wt v: \N_{>0} \to \mc V$ via successively choosing vectors as follows:
	\begin{align*}
		\wt v(1) &= \mbox{a vector appearing infinitely often in the set $\{v^{(e)}(1)\: | \: e\gg 0\}$}.\\
		\wt v(n) &= \mbox{a vector appearing infinitely often in the set $\{v^{(e)}(n)\: | \: e\gg 0,\ v^{(e)}(i) = \wt v(i) \ \forall i<n\}$}.
	\end{align*}
	By design, for every $e$, the restriction $\wt v|_{[e]}$ agrees with the restriction $v^{(f)}|_{[e]}$ for some $f>e$ (in fact, for infinitely many such $f$). This ensures that the $e$th partial sum of $\sum_{i=1}^\infty \wt v(i)\frac{1}{p^i}$ matches the $e$th partial sum of the vector corresponding to $v^{(f)}$, which in particular is a lower bound for~$\mu$.
\end{proof}

Now that we have a formula for this associated $p$-body, we can use it to give a description of the $p$-stabilization of the minimal generator system.
\begin{proof}[Proof of \Cref{minl-system-stabilization}]
	\label{proof-minl-system-stabilization}
	As in the setup of \Cref{minl-system-shape}, let
	\[
		\mc W = \left\{
			\sum_{i=1}^\infty v(i) \frac{1}{p^i} \: \bigg| \: v: \N_{>0} \to \mc V
			\right\} \subset \R^d.
	\]
	Then by \Cref{minl-system-shape},
	\[
		\Delta(\mf a_\bullet) = \left\{
   \mu \in \left( \Z[1/p]\right)^d \: \big| \: \exists\, \omega \in \mc W \textrm{ s.t. } \mu \geq \omega
			\right\}.
	\]

	Since $\mf a_e^{\Delta(\mf a_\bullet)} = \pstab{\mf a}_e$ by \Cref{shape-stabilization-correspondence}, it suffices to understand the vectors $\beta\in \Z^d$ with $\frac{1}{p^e}\beta \in \Delta(\mf a_\bullet)$, i.e., to understand the vectors $\lceil p^e\alpha\rceil$ for $\alpha \in \Delta(\mf a_\bullet)$. In particular, it suffices to understand our generators from $\mc W$, so take
	\[
		\left\lceil \sum_{i=1}^\infty v(i)\frac{p^e}{p^i} \right\rceil 
		= \sum_{i=1}^e v(i) p^{e-i} + \left\lceil \sum_{i=1}^\infty v(e+i)\frac{1}{p^i} \right\rceil.
	\]
	The term outside of the ceiling corresponds to monomials in $\prod_{i=0}^{e-1} I^{[p^i]}$, so we only need to understand the ceiling term, which is of the form $\lceil \omega \rceil$ for $\omega \in \mc W$. 
	Since there are only finitely many vectors in $\mc V$, we can factor $\omega = \sum_\nu c_\nu \nu$, where $\sum_\nu c_\nu = \sum_{i=1}^\infty \frac{1}{p^i} = \frac{1}{p-1}$ as desired.
\end{proof}

For the next $F$-graded system, we will take another approach and use our already-computed $p$-stabilization to find the associated $p$-body.

\begin{thm}
	\label{colon-system-shape}
	For any fixed monomial ideal $I\subset S$, if $\mf a_e = I^{[p^e]}:I$, then 
	\[
		\Delta(\mf a_\bullet)=\log I.
	\]
\end{thm}
\begin{proof}
	We know from \Cref{colon-system-stabilization} that $\pstab{\mf a}_e = I^{[p^e]}$, and so by \Cref{shape-stabilization-correspondence}, it suffices to show that $\Delta(I^{[p^\bullet]})=\log I$. 

	But then $\log I = \frac{1}{p^e} \log I^{[p^e]}$, and so 
	\[
		\bigcap_{f>0} \ \bigcup_{e\geq f} \frac{1}{p^e} \log I^{[p^e]} = \log I
	\]
	as desired.
\end{proof}

Finally, we will compute the associated $p$-body of the rounding system specifically when starting with the homogeneous maximal ideal. 

\begin{thm}
	\label{maxl-t-system-shape}
	Let $\mf a_e = \mf m^{\lceil t(p^e-1)\rceil}$ for $t\in \R_{\geq 0}$ and $\mf m$ the homogeneous maximal ideal in $S$. Then
	\[
		\Delta(\mf a_\bullet) 
		= \left\{ \alpha \in \left(\N[\frac{1}{p}]\right)^{d}\: \Big| \: |\alpha|\geq t\right\}.
	\]
\end{thm}
\begin{proof}
	Suppose that $\alpha \in (\N[1/p])^d$ and $|\alpha|\geq t$. Choose $E$ such that $p^E\alpha \in \N^d$. Now for all $e\geq E$, we have $p^e|\alpha| \geq p^e t$ and $p^e \alpha \in \N^d$, so that 
	\[
		|p^e\alpha| \geq \lceil p^et\rceil \geq \lceil (p^e-1)t\rceil,
	\]
	and in particular, $p^e\alpha \in \log \mf m^{\lceil (p^e-1)t\rceil}$, so that $\alpha \in \Delta(\mf a_\bullet)$.

	Conversely, suppose that $\alpha \in(\N[1/p])^d$ but $|\alpha | < t$. Now choose $F$ such that $|\alpha|< t - \frac{t}{p^F}$. Then for all $e\geq F$, we have
	\[
		|\alpha| < t - \frac{t}{p^F} \leq t(1-\frac{1}{p^e}) 
		\, \implies \,
		p^e|\alpha| < (p^e-1)t\leq \lceil (p^e-1)t\rceil,
	\]
	and in particular $p^e\alpha \notin \log \mf m^{(p^e-1)t\rceil}$ for all such $e$, so that $\alpha \notin \Delta(\mf a_\bullet)$.
\end{proof}

Again, this allows us to compute the $p$-stabilization of this system as well:

\begin{proof}[Proof of \Cref{maxl-t-system-stabilization}]
	\label{proof-maxl-t-system-stabilization}
	We simply note that $\frac{1}{p^e}\alpha\in \Delta(\mf a_\bullet)$ if and only if $\frac{1}{p^e}|\alpha| \geq t$ if and only if $\alpha \geq t p^e$. 
	Since $|\alpha|\in \N$, this gives the desired ceiling statement.
\end{proof}

\section{Numerical Properties}
\label{sec:numerical-general}

We finish by pointing out some natural connections to an invariant of $F$-graded systems.

\begin{defn}
\label{def:cartier-multiplicity}
Let $R$ be a ring of prime characteristic $p$ with $\dim R = d$ which is local or standard graded with (homogeneous) maximal ideal $\mf m$, and let $\mf a_\bullet$ be an $F$-graded system of $R$ which is \emph{eventually $\mf m$-primary}.
The \emph{volume} of $\mf a_\bullet$ is
\[
\cmult(\mf a_\bullet) := \lim_{e\to\infty}\frac{\ell(R/\mf a_e)}{p^{ed}}.
\]
\end{defn}
This limit is analogous to the Hilbert-Kunz multiplicity (first considered by Kunz in \cite{Kunz.76}, and studied in-depth by Monsky \cite{Monsky.83}) and to the $F$-signature (introduced by this name by Huneke and Leuschke in \cite{Huneke+Leuschke.02}, but not shown to exist until Tucker's work a decade later \cite{Tucker.12}). In fact this volume is an extension of the notion of the volume of a $p$-family, introduced in \cite{Hernandez+Jeffries.18}, which already encompasses these two examples of the Hilbert-Kunz multiplicity and the $F$-signature.
When working with $p$-families, Hern\'andez and Jeffries have completely characterized the rings for which this volume is always guaranteed to exist:

\begin{thm}[{\cite[Thm.~1.2]{Hernandez+Jeffries.18}}]
Let $(R,\mf m)$ be a local ring of prime characteristic $p>0$ with $\dim R=d$. Then $\cmult(\mf a_\bullet)$ exists for every $p$-family $\mf a_\bullet$ of $\mf m$-primary ideals of $R$ if and only if the $R$-module dimension of the nilradical of the completion of $R$ is less than $d$.
\end{thm}

Even for the case of $F$-graded systems, leveraging pre-existing work on the Hilbert-Kunz multiplicity and $F$-signature can show cases when this volume exists:

\begin{prop}
Let $(R, \mf m, k)$ be an $F$-finite local domain of dimension $d$, and let $\mf a_\bullet$ be an $F$-graded sequence of ideals such that $\mf m^{[p^e]}\subseteq \mf a_e$ for all $e$ and such that $\mf a_1\neq 0$. Then $\cmult(\mf a_\bullet)$ exists.
\end{prop}
\begin{proof}
The key component of this proof is \cite[Thm.~4.3]{Polstra+Tucker.18}, which shows that in fact for \emph{any} sequence of ideals $I_\bullet$ and constant $0\neq c \in R$ with $\mf m^{[p^e]}\subseteq I_e$ and $cI_e^{[p]}\subseteq I_{e+1}$ for all $e\in \N$, the limit $\lim_{e\to \infty} \frac{\ell_R(R/I_e)}{p^e}$ exists. In our case, the additional requirement that $\mf a_\bullet$ be $F$-graded and $\mf a_1\neq 0$ means we can simply take $c$ to be any non-zero element of $\mf a_1$, since then
\[
\mf a_e^{[p]}c \subseteq \mf a_e^{[p]}\mf a_1\subseteq \mf a_{e+1}
\]
as desired.
\end{proof}

Beyond existence, of particular interest to us is how this invariant relates to the associated $p$-body, and more generally how it relates the $p$-stabilization.
In \cite[Conj.~IV.5.4]{Brosowsky.24}, we conjectured that for monomial $F$-graded systems in a polynomial ring, the volume of the system is the same as the volume of the complement of the associated $p$-body. Since then, Das and Meng have proven this conjecture, and in fact have done so in a more general setting! See \cite{Das+Meng.26} for their full setup involving OK domains and $F$-graded systems which are bounded below linearly. In the setting of this paper, their result simplifies to the following:

\begin{thm}[{\cite[Thm.~7.15]{Das+Meng.26}}]
Let $S=k[x_1,\ldots, x_d]$ for an $F$-finite field $k$ and let $\mf a_\bullet$ be an $F$-graded system. If $\mf a_1$ is an $\mf m$-primary ideal, then
\[
\vol(\mf a_\bullet) = \vol(\R^d_{\geq 0}\setminus \Delta(\mf a_\bullet)).
\]
\end{thm}

\begin{rmk}
The original statement of their result references the bounded below linearly (BBL) property, meaning that there exists a $c$ such that for all $e$, we have $\mf m^{cp^e} \subset \mf a_e$. In the polynomial ring setting, if $\mf a_\bullet$ is $F$-graded and $\mf a_1$ is $\mf m$-primary then $\mf a_\bullet$ is automatically BBL. 
We can see this as follows: let $n$ be such that $\mf m^n\subset \mf a_1$. Then
\[
\mf a_e \supset \prod_{i=0}^{e-1} \mf a_1 ^{[p^i]} \supset \prod_{i=0}^{e-1} (\mf m^n)^{[p^i]} \supset \prod_{i=0}^{e-1} \mf m^{\binom{n+d-1}{n}(p^i-1)+1}
= \mf m^{\sum_{i=0}^{e-1}(\binom{n+d-1}{n}(p^i-1)+1)}.
\]
This follows by first using \Cref{minl-sys-containment}, and then using the pigeonhole principle to show that for an ideal $I$ with $g$ generators, $I^{g(q-1)+1}\subset I^{[q]}$. Now the power on $\mf m$ simplifies to $e+\binom{n+d-1}{n}\left(\frac{p^e-1}{p-1} - e\right)$, which limits to $\frac{1}{p-1}\binom{n+d-1}{n}$ if we scale by $1/p^e$ and take $e\to \infty$. In particular, there exists some fixed $N$ such that taking $c = N + \frac{1}{p-1}\binom{n+d-1}{n}$ will work for our BBL bound, and which explains our omission of the BBL hypothesis in our restatement of Das and Meng's result for polynomial rings.
\end{rmk}

Finally, given the $p$-body/$p$-stabilization correspondence (\Cref{shape-stabilization-correspondence}), the above result immediately gives the following corollary:
\begin{cor}
For an $\mf m$-primary monomial $F$-graded system $\mf a_\bullet$ in an $F$-finite polynomial ring, $p$-stabilization preserves multiplicity, i.e.,
\(
\cmult(\mf a_\bullet) = \cmult(\pstab{\mf a}_\bullet).
\)
\end{cor}

\printbibliography

\end{document}